\documentclass[11pt,a4paper]{amsart}
\usepackage{amsaddr}
\usepackage{enumerate}
\usepackage{amsfonts,amssymb}
\usepackage{amsthm,amsmath,amstext,xfrac}
\usepackage{tikz}
\usepackage{hyperref}
 \usetikzlibrary{matrix,arrows,calc,decorations.pathreplacing,decorations.markings}
 \usetikzlibrary{decorations.pathmorphing}
 \usetikzlibrary{positioning}

\usepackage{cancel,xspace}
\usepackage{epsfig,fancybox}
\usepackage{graphicx,mathrsfs,ifpdf}
\usepackage{setspace}
\usepackage{cleveref}
\usepackage{enumerate}
\usepackage[colorinlistoftodos,prependcaption,textsize=tiny]{todonotes}

\newtheorem{theorem}{Theorem}

\newtheorem{corollary}[theorem]{Corollary}

\newtheorem{conjecture}[theorem]{Conjecture}

\theoremstyle{definition}

\newtheorem{question}[theorem]{Question}

\newtheorem*{acknowledgement}{Acknowledgments}

\numberwithin{equation}{section}

\newcommand{\claimproofend}{\hspace*{.1mm}\hspace{\fill}}

\begin{document}
\title{Antifactors in bipartite multigraphs}
\author{Louis Esperet}
\address{Laboratoire G-SCOP (CNRS, Univ. Grenoble Alpes),
  Grenoble, France}
\email{louis.esperet@grenoble-inp.fr}
\date{\today}

\begin{abstract}
  Let $G$ be a $q$-regular bipartite graph with bipartition
  $(U,V)$. It was proved by Lu, Wang, and Yan in 2020 that $G$ has a
  spanning subgraph $H$ such
that each vertex of $U$ has degree 1 in $H$,
and each vertex of $V$ has degree distinct from 1 in $H$. We extend
the result to multigraphs, under the condition
that $q$ is a prime power and the number of perfect matchings of $G$
is not divisible by $q$. The condition on the number of perfect
matchings is necessary for multigraphs.

We conclude with a conjecture on the limiting distribution of the number
of perfect matchings modulo $q$ in a random bipartite $q$-regular
graph.\\

\end{abstract}


\maketitle

We need the following classical result (see \cite{Alon}).

\begin{theorem}[Combinatorial Nullstellensatz]\label{CN}
Let $\mathbb{F}$ be an arbitrary field, and let $f=f(x_1,\ldots,x_n)$
be a polynomial in $\mathbb{F}[x_1,\ldots,x_n]$. Suppose
that the  total degree
of $f$ is at most $\sum_{i=1}^n t_i$, where  each $t_i$ is  a  nonnegative
integer,  and  suppose  that the coefficient of the monomial
$\prod_{i=1}^n x_i^{t_i}$ is non-zero.  Then, if $S_1,\ldots,S_n$ are
subsets of $\mathbb{F}$ with $|S_i|> t_i$, there is $(s_1,\ldots,s_n)
\in S_1\times \ldots \times S_n$ so that $f(s_1,\ldots,s_n)\ne 0$.
\end{theorem}

Taking $\mathbb{F}$ to be $\mathrm{GF}(q)$, the Galois field on $q$
elements, we obtain the following simple corollary.

\begin{corollary}\label{c:CN}
Let $q$ be a prime power, and let $f=f(x_1,\ldots,x_n)$
be a polynomial in $\mathrm{GF}(q)[x_1,\ldots,x_n]$. Suppose
that the  total degree
of $f$ is at most $(q-1)n$, and  suppose  that the coefficient of the monomial
$\prod_{i=1}^n x_i^{q-1}$ is non-zero.  Then, there is $(s_1,\ldots,s_n)
\in \mathrm{GF}(q)^n$ so that $f(s_1,\ldots,s_n)\ne 0$.
\end{corollary}

For a vertex $v$ in a multigraph $G$, the degree of $v$ in $G$ is denoted
by $d_G(v)$.
Let
 $\mathrm{pm}(G) $ denote the number of perfect matchings of a multigraph 
 $G$.

 \begin{theorem}\label{th:main}
   Let $q$ be a prime power and let $G$ be a $q$-regular bipartite multigraph, with bipartition $(U,V)$,
and assume that each vertex $v\in V$ is assigned an arbitrary value
$\alpha(v)\in\{0,1,\ldots,q-1\}$. If  $\mathrm{pm}(G) \ne 0\pmod q$, then $G$ has a
spanning subgraph $H$ such that for each vertex $u \in U$, $d_H(u)=1$, and for each  vertex $v \in V$, $d_H(v)\ne\alpha(v) \pmod q$.
\end{theorem}

\begin{proof}
Since $G$ is $q$-regular and bipartite, it has a $q$-edge-coloring $c$. It
will be convenient to consider the $q$ colors used by $c$ as the
elements of $\mathrm{GF}(q)$, the Galois field on $q$ elements. 
For each element $i\in \mathrm{GF}(q)$, let
$g_{i}(x)=1-(x-i)^{q-1}$ (considered as a polynomial over $\mathrm{GF}(q)$ with
a single variable $x$). Note that 
\begin{equation*}
  g_{i}(x)  =
    \begin{cases}
      0 & \text{if $x \ne i$}\\
      1 & \text{if $x =i$.}
        \end{cases}       
      \end{equation*}
      We
associate to each vertex $u\in U$ a variable $x_u \in \mathrm{GF}(q)$,
and we define the polynomial $f$ over $\mathrm{GF}(q)$ with variables $(x_u)_{u \in U}$
  as follows:
$$f=\prod_{v \in V}\left[\left( \sum_{uv \in
      E(G)}g_{c(uv)}(x_u)\right)-\alpha(v)\right],$$ where the sum
ranges over all the edges incident to $v$  (multiple edges
 appear multiple times in the sum), and $c(uv)$ denotes the
color of the edge $uv$ in the $q$-edge-coloring $c$.

 Since each polynomial $g_{i}$ has degree $q-1$, $f$ has degree at most $(q-1)|V|=(q-1)|U|$
 (the equality $|U|=|V|$ follows from the fact that $G$ is regular and
 bipartite). Observe that the linear coefficient of the monomial $\prod_{u\in U}x_u^{q-1}$ in
 $f$ is precisely $(-1)^{|V|}\mathrm{pm}(G) \pmod q$. Since $\mathrm{pm}(G) \ne 0 \pmod q$, it
 follows from Corollary~\ref{c:CN} that for each $u\in U$ there is a
 variable $x_u\in \mathrm{GF}(q)$, such that $f$ is non-zero when
 evaluated in $(x_u)_{u \in U}$. For each vertex $u\in U$, let $e_u$
 be the unique edge incident to $u$ such that $c(e_u)=x_u$, and let $H$ be
 the subgraph of $G$ induced by the edges $(e_u)_{u\in U}$. Then
 by definition each vertex of $U$ has degree 1 in $H$. Since $f$ is
 non-zero on $(x_u)_{u \in U}$, we have that for each $v\in V$, $$\sum_{uv \in
      E(G)}g_{c(uv)}(x_u)\ne \alpha(v)\pmod q.$$ It remains to
    observe that the sum appearing on the left hand
    side counts precisely the degree of $v$ in $H$.
  \end{proof}

   We emphasize that since the graph $G$ is $q$-regular, the
    condition $d_H(v)\ne\alpha(v)
    \pmod q$ is only strictly stronger than $d_H(v)\ne\alpha(v)$ if $\alpha(v)=0$. In this case, this
    is equivalent to say  that $v$ must have degree distinct from $0$
    and $q$ in $H$.

\medskip

 We remark that the condition $\mathrm{pm}(G)\ne 0 \pmod q$ is necessary: examples include the unique $q$-regular bipartite multigraph on
    2 vertices, and cycles of length $2 \pmod 4$ in which the edges of
    a perfect
    matching are replaced by multiple edges with multiplicity $q-1$. These multigraphs have a number of
    perfect matchings that is divisible by $q$, and do not have a subgraph $H$
    as in Theorem~\ref{th:main} if $\alpha(v)=1$ for each vertex $v$
    in one part of the bipartition. It was proved in~\cite{LWY} (see
    also \cite{Seb20}) that
    in this specific case (when $\alpha$ is the function
    assigning the value 1 to each
    vertex of the domain), such examples necessarily contain multiple
    edges.

    \medskip
    
    This is reminiscent of the Berge-Sauer conjecture, which stated that
    every 4-regular graph contains a 3-regular subgraph. This conjecture was proved
    by Tashkinov \cite{Tas84}, and was known to be false for
    multigraphs. On the other hand, Alon Friedman and Kalai \cite{AFK84} 
    proved that every 4-regular multigraph \emph{plus an edge}
    contains a 3-regular subgraph, using the polynomial method. So in
    both situations we have a problem with a combinatorial solution,
    whose conclusion only holds for simple graphs, while polynomial
    techniques show that a slighlty weaker
    statement holds for multigraphs.

     \medskip

Given the necessary condition in the statement of Theorem
\ref{th:main}, it is natural to investigate the typical residue of
the number of perfect matchings modulo $q$ for bipartite $q$-regular multigraphs.


\begin{conjecture}\label{conj}
For any odd integer $q\ge 3$ there is a real
$\varepsilon>0$ such that for any $i\in \{0,1,\ldots,q-1\}$, the number of perfect matchings
modulo $q$ of an $n$-vertex random $q$-regular bipartite (multi)graph
($n$ even and sufficiently large)
is equal to $i$ with probability at least $\varepsilon$.
\end{conjecture}


The reason we restrict the conjecture above to odd values of $q$ is a
classical result of Little \cite{Lit72}, stating that every $q$-regular
bipartite graph with $q$ even has an even number of perfect
matchings.

It is plausible that the limiting distribution of the number of perfect matchings
modulo $q$ of an $n$-vertex random $q$-regular bipartite (multi)graph
($q$ odd),
when $n\to\infty$ ($n$ even)  is the uniform distribution over $\{0,1,\ldots
,q-1\}$. Exact computations with $q=3$ and $n= 26$  seem
to confirm this. There are 245627	 (non-isomorphic) connected cubic bipartite graphs on 26 vertices
\cite{Bri96,house}, and it can be checked that the proportion
of these graphs with
$\mathrm{pm}(G)= 0,1,2 \pmod 3$ is 0.366,
0.314, and 0.321
respectively. Observe that if in a cubic graph $G$ there is a vertex $u$ such that
for any pair of edges $e,f$ incident to $u$, there is an automorphism
mapping $e$ to $f$, then the number of perfect matchings is 
divisible by 3. The proportion of graphs satisfying this type of
property is significant when the number of vertices is small (which
might explain the larger proportion $0.366 >\tfrac13$ in the case $n=26$ above). However
this proportion is vanishingly small as $n\to \infty$.


\medskip

 See \cite{AHK11} for some results on the residue modulo $q$ of the number of
 (non necessarily perfect) matchings in random trees. In this case the
 situation is very different from above, and it can be shown that
 a.a.s.\ the
 number of matchings is divisible by $q$.




 \medskip

 We conclude with a final remark. Recall that Lu, Wang, and Yan
 \cite{LWY} proved
 that  any bipartite $q$-regular graph $G$ with bipartition $(U,V)$ has a
  spanning subgraph $H$ such
that each vertex of $U$ has degree 1 in $H$,
and each vertex of $V$ has degree distinct from 1 in $H$. This
corresponds to the conclusion of Theorem \ref{th:main} with
$\alpha(v)=1$ for any $v\in V$. Observe that when  $G$ is a
multigraph, instead of requiring that
the number of perfect matchings of $G$ is not divisible by $q$ and
applying Theorem \ref{th:main}, it is
enough to assume that $G$ contains a spanning 3-regular subgraph $G'$
whose number of perfect matching is not divisible by 3 (since in this
case we can apply Theorem \ref{th:main} to $G'$ with $\alpha(v)=1$ for
any $v\in V$, and obtain the desired spanning subgraph $H$). Let us
say that a $q$-regular bipartite multigraph $G$ is \emph{bad} if it
does not contain a 3-regular spanning subgraph whose number of perfect
matchings is not divisible by 3.

\begin{question}
  Can we completely describe the bad
  graphs? Is there a polynomial time algorithm for detecting whether a
  multigraph is bad? 
\end{question}

\begin{acknowledgement}
I would like to thank Zolt\'an Szigeti, Andr\'as Seb\H{o}, and
St\'ephan Thomass\'e  for the discussions, and Prajit Adhikari for
pointing out that the original version of Conjecture \ref{conj} was
false for even values of $q$.
\end{acknowledgement}

\section*{Statements and Declarations}

\noindent {\it Funding.} This work was partially supported by ANR Projects GATO
(\textsc{anr-16-ce40-0009-01}) and GrR (\textsc{anr-18-ce40-0032}).\\

\noindent {\it Competing interests.} The author has no relevant financial or non-financial interests to disclose.\\
  
\noindent {\it Data availability.} The short Sage code used to compute the
number of perfect matchings modulo 3 in the graphs from \cite{house} is available from the corresponding author on request.

\end{document}